\documentclass[12pt]{article}
\usepackage{amsmath, amscd, amsfonts, amssymb, amsthm, tikz, times}
\RequirePackage[colorlinks,citecolor=blue,urlcolor=blue,linkcolor=blue]{hyperref}
\newtheorem{theorem}{Theorem}[section]

\newtheorem{lemma}[theorem]{Lemma}

\newtheorem{definition}[theorem]{Definition}

\def\1{1}
\def\2{1_{[0,A_t]}}

\def\sbs{{S$\beta$S}}
\def\sas{{S$\al$S}}
\def\supp{\text{supp}}
\def\ed{\stackrel{d}{=}}
\def\t1{\tilde{1}}

\def\ff{\infty}

\def\aa{\alpha}

\def\Om{\Omega}

\def\al{\alpha}
\def\bb{\beta}
\def \bl{\begin{lemma}}
\def \el{\end{lemma}}

\def \bt{\begin{theorem}}
\def \et{\end{theorem}}

\def\be{\begin{equation}}
\def\ee{\end{equation}}
\def\bea{\begin{eqnarray}}
\def\eea{\end{eqnarray}}
\def\nn{\nonumber}
\def\eps{\epsilon}
\def\({\left(}
\def\){\right)}
\def\[{\left[}
\def\]{\right]}

\def\lb{\left|}
\def\rb{\right|}

\def\sign{\text{sign}}

\def\N{\mathbb{N}}
\def\R{\mathbb{R}}
\def\Z{\mathbb{Z}}

\def\F{\mathcal{F}}

\def \P{\mathbf{P}}
\def \E{\mathbf{E}}

\def \BB{{\cal B}}

\begin{document}

\title{Indicator fractional stable motions}
\author{Paul Jung
\footnote{Supported in part by Sogang University research grant 200910039}\\Department of Mathematics\\
  Sogang University\\
  Seoul, South Korea 121-742
   }

\date{\today} \maketitle
\abstract{Using the framework of random walks in random scenery, Cohen and Samorodnitsky (2006) introduced a family of symmetric $\aa$-stable motions called local time fractional stable motions.  When $\aa=2$, these processes are precisely fractional Brownian motions with $1/2<H<1$. Motivated by random walks in alternating scenery, we find a ``complementary" family of symmetric $\aa$-stable motions which we call indicator fractional stable motions. These processes are complementary to local time fractional stable motions in that when $\aa=2$, one gets fractional Brownian motions with $0<H<1/2$. \\
}

\\


\tableofcontents

\section{Introduction}
There are a plethora of integral representations for Fractional Brownian motion (FBM) with Hurst parameter $H\in(0,1)$, and not surprisingly there are
several generalizations of these integral representations to stable processes.  These generalizations are often called fractional symmetric $\aa$-stable (\sas) motions, with $0<\aa<2$, and they can be considered analogs of FBM.
Two common fractional \sas\ motions include
linear fractional stable motion (L-FSM) and real harmonizable fractional stable motion (RH-FSM).

In \cite{CS}, a new generalization of FBM, $H>1/2$, called {\it local time fractional stable motion} (LT-FSM) was introduced. LT-FSM is particularly interesting because it is a subordinated process (this terminology is taken from Section 7.9 of \cite{samorodnitsky1994stable} and should not be confused with subordination in the sense of time-changes).  Subordinated processes are processes constructed from integral representations
with random kernels, or said another way, where the stable random measure (of the integral representation) has a control measure related in some way to a probability measure of some other stochastic process (see Section \ref{sec:model} below). We note that subordinated processes are examples of what are known in the literature as doubly stochastic models.

In this work we introduce another subordinated process which can be considered a natural extension of LT-FSM to $H<1/2$. The processes we consider have random kernels of a very simple type, namely the indicator function
\be\label{def:signed version}
\1_{[A_0,A_t]}(x) \quad ([A_s,A_t]:=[A_t,A_s] \text{ if } A_t<A_s)
\ee
with respect to some self-similar stationary increment (SSSI) process $A_t$.
As such we call these processes {\it indicator fractional stable motions} (I-FSM).

I-FSM's relation to LT-FSM comes from the idea that the indicator function of a real-valued process $A_t$ can be thought of as an alternating version of the local time of $A_t$ in the following way.  Suppose $S_n$, with $S_0=0$, is a discrete-time simple random walk on $\Z$. If $e$ is the edge between $k$ and $k+1$, then the discrete local time of $S_n$ at $e$ is the total number of times $S_n$ has gone from either $k$ to $k+1$ or from $k+1$ to $k$, up to time $n$. Now, instead of {\it totaling} the number of times $S_n$ crosses over edge $e$, one can consider the {\it parity} of the number of times $S_n$ crosses $e$ up to time $n$.
The parity of the discrete local time at edge $e$ up to time $n$ is odd if and only if $e$ is between $0$ and $S_n$. Thus, heuristically, the edges which contribute to an ``alternating local time" are those edges which lie between $0$ and $A_t$. This heuristic is discussed more rigorously in \cite{JM}.

We can generalize the motivational discrete model to all random walks on $\Z$.  In this case, when $S_n$ goes from $x$ to $y$ on a given step, it ``crosses" all edges in between. In terms of the discrete local time, we heuristically think of the random walk as having spent a unit time at {\it all} edges between $x$ and $y$ during that time-step.  

The first question one must ask is: are these new stable processes a legitimate new class of processes or are they just a different representation of L-FSMs and/or RH-FSMs?  Using characterizations of the generating flows for the respective processes (see Section \ref{sec:background} below), \cite{CS} showed that the class of LT-FSMs is disjoint from the classes of RH-FSMs and L-FSMs.  Following their lead, we use the same characterizations to show that when the (discretized) subordinating process $\{A_n\}_{n\in\N}$ is recurrent, the class of I-FSMs is also disjoint from the two classes, RH-FSMs and L-FSMs. Since I-FSMs and LT-FSMs have disjoint self-similarity exponents when $1<\aa<2$, these two classes of processes are also disjoint when $1<\aa<2$. For $\aa<1$, the class of I-FSMs has a strictly larger self-similarity range than the class of LT-FSMs.

The outline of the rest of the paper is as follows. In the next section we define I-FSMs and show that they are \sas-SSSI processes. In Section \ref{sec:background} we give the necessary background concerning generating flows and characterizations with respect to them.  In Section \ref{sec:ergodic}, we give the classification of I-FSMs according to their generating flows along with a result on the mixing properties of the stable noise associated with an I-FSM.

\section{Indicator fractional stable motions}\label{sec:model}
Let $m$ be a $\sigma$-finite measure on a measurable space $(B,\BB)$, and let $$\BB_0 = \{A\in \BB: m(A)<\infty\}.$$
\begin{definition}\label{def:random measure}
A \emph{\sas\ random measure} $M$ with \emph{control measure} $m$ is a $\sigma$-additive set function on $\BB_0$ such that for all $A_i\in \BB_0$
\begin{enumerate}
\item $M(A_1)\sim S_{\aa}(m(A_1)^{1/\aa})$
\item $M(A_1)$ and $M(A_2)$ are independent whenever $A_1 \cap A_2 =\emptyset$
\end{enumerate}
where $S_\aa(\sigma)$ is a \sas\ random variable with scale parameter $\sigma$ (see Section 3.3 of \cite{samorodnitsky1994stable} for more details).
\end{definition}
Another way to say the second property above is to say that $M$ is {\it independently scattered}.

For context, let us first define LT-FSM. Throughout this paper
$$\lambda := \text{Lebesgue measure on }\R.$$
Let $(\Omega',\F',\P')$ support a subordinating process $A_t$. $A_t$ is either a FBM-${H'}$ or a \sbs-Levy motion, $\beta\in(1,2]$, with jointly continuous local time $L_A(t,x)(\omega')$. By self-similarity, $A_0=0$ almost surely. Suppose a \sas\ random measure $M$ with control measure
$\P'\times\lambda$ lives on some other probability space $(\Omega,\F,\P)$. An LT-FSM is a process
\be\label{def:LTFSM}
X^{H}_A(t) := \int_{\Om'}\int_{\R} L_A(t,x)(\omega') M(d\omega', dx),\ t\ge 0,
\ee
where $X^{H}_{A}(t)$ is a \sas-SSSI process with self-similarity exponent $H = 1-{H'}+{H'}/\aa$ and $H'$ is the self-similarity exponent of $A_t$
(see Theorem 3.1 in \cite{CS} and Theorem 1.3 in \cite{DG}).

We now define I-FSM which is the main subject of this work. Let $(\Omega',\F',\P')$ support $A_t$, a non-degenerate \sbs-SSSI process with $\beta\in(1,2]$ and self-similarity exponent ${H'}\in(0,1)$ (again by self-similarity $A_0=0$ almost surely). Suppose a \sas\ random measure $M$ with control measure
$\P'\times\lambda$ lives on some other probability space $(\Omega,\F,\P)$.


An {\bf indicator fractional stable motion} is a process
\be\label{def:IFSM}
Y^{H}_{A}(t) := \int_{\Om'}\int_{\R} 1_{[0,A_t(\omega')]}(x) M(d\omega', dx),\ t\ge 0.
\ee
A nice observation is that the finite dimensional distributions of the process do not change if we replace the kernel $1_{[0,A_t(\omega')]}(x)$ with $\sign(A_{t}(\omega'))1_{[0,A_{t}(\omega')]}(x)$:
\bea\label{eqn:sign}
&&\sum_{j=1}^n\theta_j\int_{\Om'}\int_{\R} \sign(A_{t_j}(\omega'))1_{[0,A_{t_j}(\omega')]}(x) M(d\omega', dx)\nonumber\\
&=& \sum_{j=1}^n\theta_j\int_{\Om'}\int_{\R^+} 1_{\{\omega':A_{t_j}(\omega')>0\}}1_{[0,A_{t_j}(\omega')]}(x) M(d\omega', dx)\nonumber\\
&& + \sum_{j=1}^n\theta_j\int_{\Om'}\int_{\R^-} -1_{\{\omega':A_{t_j}(\omega')<0\}}1_{[0,A_{t_j}(\omega')]}(x) M(d\omega', dx)\nonumber\\
&\ed&\sum_{j=1}^n\theta_j\int_{\Om'}\int_{\R} 1_{[0,A_{t_j}(\omega')]}(x) M(d\omega', dx).
\eea
where the last line holds since $M$ is both symmetric and independently scattered.

The reason that this is helpful is because the equality
\be\label{sign_rep}
\sign(A_t)\1_{[0,A_t]}(x) = (A_t-x)_+^0 - (-x)_+^0
\ee
makes it intuitively clear that the increments of $Y^{H}_{A}(t)$ are stationary.

We note that both LT-FSM and I-FSM can technically be extended to the case where $A_t$ has self-similarity exponent $H'=1$. In these degenerate cases, the kernels for LT-FSM and I-FSM coincide becoming the non-random family of functions $\{1_{[0,t]}\}_{t\ge 0}$ thereby giving us
$$\int_{\R} 1_{[0,t]} \,M(dx),\ t\ge 0 .$$
These are the \sas\ Levy motions with $\aa\in(0,2)$.
\begin{theorem}\label{thm:SSSI}
The process $Y^{H}_A(t)$ is a well-defined \sas-SSSI process with self-similarity exponent $H=H'/\aa$.
\end{theorem}
\begin{proof}
We start by noting that
\bea
\nn\int_{\Om'}\int_{\R} |1_{[0,A_t(\omega')]}(x)|^\aa \, dx \,\P'(d\omega') &=& \E'\int_{\R} 1_{[0,A_t(\omega')]}(x)\, dx\\
&=& \E'|A_t|<\infty
\eea
where the finite expectation follows since $A_t$ is a \sbs\ process with $\beta>1$.
This shows that $Y^{H}_A(t)$ is a well-defined \sas\ process (see Section 3.2 of \cite{samorodnitsky1994stable} for details).

Recall that the control measure for $M$ is $\P'\times\lambda$. Using the alternative kernel given in \eqref{eqn:sign}, by Proposition 3.4.1 in \cite{samorodnitsky1994stable} we have
for $\theta_j\in\R$ and times $t_j,s_j\in\R^+$:
\bea\label{eq:ST char func def}
&&\E \exp\(i\sum_{j=1}^k \theta_j (Y^{H}_A(t_j)-Y^{H}_A(s_j))\) \nonumber\\
&=& \exp\(-\int_\R \E'\lb\sum_{j=1}^k\theta_j\cdot\sign(A_{t_j}-A_{s_j})\,\1_{[A_{s_j},A_{t_j}]}(x)\rb^\aa\,dx\).
\eea
Note that if we had not used the alternative kernel given in \eqref{eqn:sign}, then the right-side above would have been more complicated.

Using \eqref{eq:ST char func def}, we have
\bea
&&\nn\E \exp\(i\sum_{j=1}^k \theta_j (Y^{H}_A(t_j+h)-Y^{H}_A(h))\)\\
&=&\nn  \exp\(-\int_\R \E'\lb\sum_{j=1}^k\theta_j\cdot\sign(A_{t_j+h}-A_{h})\1_{[A_{h},A_{t_j+h}]}(x)\rb^\aa\,dx\)\\
&=&\nn \exp\(-\int_\R \E'\lb\sum_{j=1}^k\theta_j\cdot\sign(A_{t_j})\1_{[0,A_{t_j}]}(x)\rb^\aa\,dx\)\\
&=& \E \exp\(i\sum_{j=1}^k \theta_j Y^{H}_A(t_j)\)
\eea
where the second equality follows since $A_t$ has stationary increments.  The above
shows that $Y^{H}_A(t)$ has stationary increments.

Using \eqref{eq:ST char func def} once more, the self-similarity of $\{A_t\}_{t\ge 0}$, and the change of variables $y=c^{-H'}x$, we obtain
\bea
\nn\E \exp\(i\sum_{j=1}^k \theta_j Y^{H}_A(ct_j)\) &=&  \exp\(-\int_\R \E'\lb\sum_{j=1}^k\theta_j\cdot\sign(A_{ct_j})\1_{[0,A_{ct_j}]}\rb^\aa\,dx\)\\
&=&\nn \exp\(-c^{H'}\int_\R \E'\lb\sum_{j=1}^k\theta_j\cdot\sign(A_{t_j})\1_{[0,A_{t_j}]}\rb^\aa\,dy\)\\
&=& \E \exp\(i\sum_{j=1}^k \theta_j c^{H'/\aa}Y^{H}_A(t_j)\)
\eea
\end{proof}

\vspace{1.35cm}
\noindent \emph{Remarks.}
\begin{enumerate}
\item For each fixed $0<\aa<2$, I-FSM is a class of \sas-SSSI processes with self-similarity exponents $H$ in the feasibility range
$0<H < 1/\aa$.
In particular, when $1<\aa<2$, this range of feasible $H$ complements that of LT-FSM which has the feasibility range $1/\aa<H<1$.
When $0<\aa<1$, the feasibility range  $0<H<1/\aa$ of I-FSM is strictly bigger than that of LT-FSM: $1<H<1/\aa$.

\item It is not hard to see that I-FSMs are continuous in probability since the subordinating process $A_t$ is SSSI and continuous in probability. However, it follows from Theorem 10.3.1 in \cite{samorodnitsky1994stable} that I-FSMs are not sample continuous.  This is intuitive since I-FSMs should have continuity properties similar to those of \sas\ Levy motions since the latter have the form
    \be
\int_{\R^+} \1_{[0,t]}(x) M(dx),\ t\ge 0
\ee
where $M$ is a \sas\ random measure with Lebesgue control measure.

\item By Theorem 11.1.1 in \cite{samorodnitsky1994stable} an I-FSM has a measurable version if and only if the subordinating process $A_t$ has a measurable version.
\end{enumerate}

\section{Background: Ergodic properties of flows}\label{sec:background}
Throughout this section we suppose that $0<\aa<2$. The general integral representations of
$\aa$-stable processes, of the type
\begin{equation} \label{e:stab.process}
X(t) = \int_E f_t(x) \; M(dx), \; t\in T
\end{equation}
 ($T=\Z$ or $\R$) are well-known (see the introduction of \cite{samorodnitsky2005null}).
Here $M$ is a \sas\ random measure on $E$ with a $\sigma$-finite
control measure $m$, and $f_t\in L^\aa(E,m)$ for each $t$.
We call $\{f_t(x)\}_{t\in T}$ a {\it spectral representation} of $\{X(t)\}$.

\begin{definition}
A measurable family of functions $\{\phi_t\}_{t\in T}$ mapping $E$
onto itself and such that
\begin{enumerate}
\item
$\phi_{t+s}(x) =\phi_t(\phi_s(x))$ for all $t,s\in T$ and $x\in
E$,
\item $\phi_0(x)=x$ for all  $x\in E$
\item $m\circ \phi_t^{-1}\sim m$
for all $t\in T$
\end{enumerate}
is called a \emph{nonsingular flow}.
A measurable family $\{a_t\}_{t\in T}$ is called a \emph{cocycle} for the flow $\{\phi_t\}_{t\in T}$
if for
every $s,t\in T$ we have
\be
a_{t+s}(x) = a_s(x)a_t(\phi_s(x))\ m\text{-a.e.}.
\ee
\end{definition}

In \cite{rosinski1995structure}
it was shown that in the case of measurable stationary
\sas\ processes one can choose the (spectral) representation in
(\ref{e:stab.process}) to be of the form
\begin{equation} \label{e:station.kernel}
f_t(x) = a_t(x) \,\left( \frac{dm\circ
\phi_t}{dm}(x)\right)^{1/\aa}  f_0\circ \phi_t(x)
\end{equation}
where $f_0\in L^\aa(E,m)$,  $\{\phi_t\}_{t\in T}$ is a nonsingular flow, and $\{a_t\}_{t\in T}$ is a cocycle, for $\{\phi_t\}_{t\in T}$, which takes values in $\{-1,1\}$.
Also, note that one may always assume the following full support condition:
\be\label{cond:support}
\supp\{f_t:t\in T\}=E.
\ee

Henceforth we shall assume that $T=\Z$ and will write $f_n$, $\phi_n$, and $X(n)$. Note that in the discrete case we may always assume measurability of the process (see Section 1.6 of \cite{aaronson1997introduction}). Given a representation of the form \eqref{e:station.kernel}, we say that $X(n)$ is generated by $\phi_n$.

In \cite{rosinski1995structure} and \cite{samorodnitsky2005null}, the ergodic-theoretic
properties of a generating
flow $\phi_n$ are related to the probabilistic properties of the \sas\ process
$X(n)$.  In particular,
certain
ergodic-theoretic properties of the flow are found to be invariant from representation
to representation.

In Theorem 4.1 of \cite{rosinski1995structure} it was shown that the {\it Hopf \,decomposition} of a flow is
a representation-invariant property of stationary \sas\ processes. Specifically,
one has the disjoint union $E=C\cup D$ where the {\it dissipative} portion $D$ is the union of all wandering sets and the {\it conservative} portion $C$ contains no wandering subset. A {\it wandering} set is one such that $\{\phi_n(B)\}_{n\in\Z}$ are disjoint modulo sets of measure zero.  Since $C$ and $D$ are $\{\phi_n\}$-invariant, one can decompose a flow by looking at its restrictions to $C$ and $D$, and the decomposition is unique modulo sets of measure zero.
A nonsingular flow $\{\phi_n\}$ is said to conservative if $m(D)=0$ and dissipative if $m(C)=0$.

The following result appeared as Corollary 4.2 in \cite{rosinski1995structure} and has been adapted to the current context:
\begin{theorem}[Rosinski]\label{thm:rosinski}
Suppose $0<\aa<2$. A stationary \sas\ process is generated by a conservative (dissipative, respectively) flow if and only if for some (all) measurable spectral representation
$\{f_n\}_{n\in\R^+}\subset L^{\aa}(E,m)$ satisfying \eqref{cond:support}, the sum
\be
\sum_{n\in \Z} |f_n(x)|^{\aa}
\ee
is infinite (finite) $m$-a.e. on $E$.
\end{theorem}

In \cite{samorodnitsky2005null}, another representation-invariant property of flows, the {\it positive-null} decomposition of stationary
\sas\ processes, was introduced.

A subset $B\subset E$ is called {\it weakly
wandering} if there is a subsequence with $n_0=0$ such that the
sets $\{\phi_{n_k} B\}_{k\in\N}$ are disjoint modulo sets of measure zero.
The null part $N$ of $E$ is the union of all weakly wandering sets, and the positive part $P$ contains no weakly wandering set.  Note that the positive part of $E$ is a subset of the conservative part, i.e. $P\subset C$.
Again, one can decompose $\{\phi_n\}$  by restricting to $P$ and ${N}$.
This decomposition is unique modulo sets of measure zero, and Theorem 2.1 of \cite{samorodnitsky2005null} states that the decomposition is representation-invariant modulo sets of measure zero. A null flow is one with $m(P)=0$ and a positive flow has $m(N)=0$.
Note that dissipative flows are automatically null flows, however in the case of conservative flows, both positive and null flows are possible.


\section{Ergodic properties of indicator fractional stable noise}\label{sec:ergodic}
Properties of a \sas-SSSI process $Y(t)$ are often deduced from its increment process $Z(n) = Y(n)-Y(n-1), n\in\N$ called a {\it stable noise}.  In this section,  we study the
ergodic-theoretic properties (which were introduced in the previous section) of {\it indicator fractional stable noise} (I-FSN) which we define as
\be
Z_A(n):= \int_{\Om'}\int_{\R} \1_{[0,A_{n}(\omega')]}(x)-\1_{[0,A_{n-1}(\omega')]}(x) M(d\omega', dx),\ n\in\N.
\ee
We note that in light of the proof of Theorem \ref{thm:SSSI}, one may deem it natural to instead use the kernel $$\sign(A_n(\omega'))\1_{[0,A_{n}(\omega')]}(x)-\sign(A_{n-1}(\omega'))\1_{[0,A_{n-1}(\omega')]}(x).$$  However, as seen in \eqref{eqn:sign}, the $\sign(A_t)$ has no affect on the distribution of the process and therefore has no affect on the distribution of its increments.

It is known that stationary \sas\ processes generated by dissipative flows are mixing \cite{surgailis1993stable}.  Concerning conservative flows,
Theorem 3.1 of \cite{samorodnitsky2005null} states that a stationary \sas\ process is ergodic if and only if it is generated by a null flow, and examples
are known of both mixing and non-mixing stationary \sas\ processes generated by conservative null flows (see Section 4 of \cite{gross1993ergodic}). Our next goal is to show that
I-FSN is mixing which implies that its flow is either dissipative or conservative null.  We first need a result which appeared as Theorem 2.7 of \cite{gross1994some}:

\begin{lemma}[A. Gross]
Suppose $X_n$ is some stationary \sas\ process, and assume $\{f_n\}\subset L^{\aa}(E,m)$ is a spectral representation of $X_n$ with respect to the control measure $m$.
Then $X_n$ is mixing if and only if for every compact $K\subset\R-\{0\}$ and
every $\eps>0$,
\be
\lim_{n\to \ff} m\{x:f_0\in K, |f_n|>\eps\}=0.
\ee
\end{lemma}

\begin{theorem}\label{thm:mixing}
Indicator fractional stable noise is a mixing process.
\end{theorem}
\begin{proof}
Using the above lemma, it suffices to show that
\be
\lim_{n\to\ff}(\P'\times\lambda)\{(\omega',x):x\in[0,A_1],x\in[A_n,A_{n+1}]\}=0,
\ee
recalling that $[A_n,A_{n+1}]:=[A_{n+1},A_{n}]$ whenever $A_{n+1}<A_n$.

Let $c_i$ be constants such that for all $M>0$, $\P'(A_1>M)<c_1M^{-\bb}$ and
\be\label{eqn:tail}
\int_M^\ff\P'(A_1>x)\,dx<c_2M^{-\bb+1}
\ee where $\bb>1$.
Also, recall that $0<H'<1$ is the self-similarity exponent of $A_t$.
We have that
\bea \nn
&&(\P'\times\lambda)\{(\omega',x):x\in[0,A_1],x\in[A_n,A_{n+1}]\}\\
&=&(\P'\times\lambda)\{(\omega',x):|x|>M, x\in[0,A_1],x\in[A_n,A_{n+1}]\}\nn\\
&&+(\P'\times\lambda)\{(\omega',x):|x|\le M, x\in[0,A_1],x\in[A_n,A_{n+1}]\}\nn\\
&\le& 2\int_{M}^{\infty} \P'(A_1>x)\, dx + (\P'\times\lambda)\{(\omega',x):|x|\le M,x\in[A_n,A_{n+1}]\} \nn\\
&\le& 2c_2M^{-\bb+1} + 2M\sup_{x\in[-M,M]}\P'\{\omega':x\in[A_n(\omega'),A_{n+1}(\omega')]\} \nn\\
&\le& 2c_2M^{-\bb+1} + 2M \P'\(\{|A_n|\le M\}\cup \{|A_{n+1}|\le M\}\)\nn\\
&&+2M \P'\(\{A_n<-M, A_{n+1}>M\} \cup\{A_n>M, A_{n+1}<-M\}\)\nn\\
&\le& 2c_2M^{-\bb+1} +4M \P'\(|A_1|\le M/n^{H'}\) + 2M\cdot 2c_1M^{-\bb}.\label{eqn22}
\eea
where the first inequality uses the symmetry of $A_1$. The second inequality uses \eqref{eqn:tail}, and the third inequality uses the fact that for $x\in[-M,M]$, the event $\{\omega':x\in[A_n(\omega'),A_{n+1}(\omega')]\}$ is contained by the event that either $A_n$ or $A_{n+1}$ is in $[-M,M]$ or that $[A_n,A_{n+1}]$ (which we defined as equivalent to $[A_{n+1},A_n]$) contains $[-M,M]$. The final inequality uses both self-similarity and stationarity of increments.

Since the right side of \eqref{eqn22} can be made arbitrarily small by choosing $M$ and then $n$ appropriately, the result is proved.
\end{proof}

%

Since I-FSN is mixing, it is generated by a flow which is either dissipative or conservative null. Our next result classifies the flow of I-FSN as conservative if almost surely
\be\label{recurr}
\limsup_{n\to\infty} A_n = +\infty \quad\text{and}\quad \liminf_{n\to\infty} A_n = -\infty \quad\text{where }n\in\N.
 \ee
 This holds, for example, when $A_t$ is a FBM or a \sbs\ Levy motion with $\beta> 1$.


\begin{theorem}
If the subordinating process $A_t$ satisfies \eqref{recurr}, then the indicator fractional stable noise, $\{Z_A(n)\}_{n\in \Z}$, is generated by a conservative null flow.
\end{theorem}


\begin{proof}
%
%
By \eqref{recurr}, we have that $\P'$-almost surely
\bea
&&\sum_{n=0}^\infty |\1_{[0,A_{n}(\omega')]}(x)-\1_{[0,A_{n-1}(\omega')]}(x) |^\aa\nn\\
&&= \sum_{n=0}^\infty 1_{[A_n(\omega'),A_{n+1}(\omega')]}(x) =\ff \quad \text{for every }x.
\eea
Hence by Theorem \ref{thm:rosinski} we have that $Z_A(n)$ is generated by a conservative flow. By Theorem \ref{thm:mixing} the flow is also null.

\end{proof}
\vspace{1.8cm}
\noindent \emph{Remarks.}
\begin{enumerate}
\item
When $A_n$ satisfies \eqref{recurr}, the fact that I-FSMs are generated by conservative null flows implies they form a class of processes which are disjoint from the class of RH-FSMs (positive flows) and disjoint from the class of L-FSMs (dissipative flows). We have already seen that the classes of I-FSMs and LT-FSMs are disjoint when $1<\aa<2$ due to their self-similarity exponents.
\item
Another useful property of conservative flows comes from Theorem 4.1 of \cite{samorodnitsky2004extreme}: If $Z_A(n)$ is generated by a conservative flow, then it satisfies the following extreme value property:
\be
n^{-1/\aa} \max_{j=1,\ldots n} Z_A(n)\stackrel{p}{\to}0.
\ee
\end{enumerate}


%


\section*{Acknowledgements}
We wish to thank an anonymous referee for a careful reading and helpful comments.


\end{document}